\documentclass[11pt,reqno]{amsart}
\usepackage{amsmath}
\usepackage{amsfonts}
\usepackage{amssymb,latexsym}
\usepackage{cite}
\usepackage[height=190mm,width=130mm]{geometry}

\theoremstyle{plain}
\newtheorem{theorem}{Theorem}

\newtheorem{corollary}{Corollary}

\theoremstyle{definition}
\newtheorem{definition}{Definition}
\theoremstyle{remark}

\numberwithin{equation}{section} 
\input{tcilatex}

\begin{document}
\title[The generalized bi-periodic Fibonacci quaternions and octonions]{The
generalized bi-periodic Fibonacci quaternions and octonions}
\author{Elif TAN}
\address{Department of Mathematics, Ankara University, Science Faculty,
06100 Tandogan Ankara, Turkey.}
\email{etan@ankara.edu.tr}
\author{Murat SAHIN}
\address{Department of Mathematics, Ankara University, Science Faculty,
06100 Tandogan Ankara, Turkey}
\email{msahin@ankara.edu.tr}
\author{Semih YILMAZ}
\address{Department of Mathematics, K\i r\i kkale University, 71450 K\i r\i
kkale, Turkey.}
\email{syilmaz@kku.edu.tr}
\subjclass[2000]{ 11B39, 05A15, 11R52}
\keywords{Quaternions, octonions, Fibonacci sequence, bi-periodic Fibonacci
sequence}

\begin{abstract}
In this paper, we present a further generalization of the bi-periodic
Fibonacci quaternions and octonions. We give the generating function, the
Binet formula, and some basic properties of these quaternions and octonions.
The results of this paper not only give a generalization of the bi-periodic
Fibonacci quaternions and octonions, but also include new results such as
the matrix representation and the norm value of the generalized bi-periodic
Fibonacci quaternions.
\end{abstract}

\maketitle

\section{Introduction}

There has been growing interest in quaternions that have been extensively
studied in both applied and theoretical sciences. In particular, quaternions
are very good at representing rotations in three-dimensional space. The
octonions are invented as an analog to the quaternions, and related to the
exceptional Lie algebra. Also they have applications in areas such as super
string theory, projective geometry, topology, and Jordan algebras. For more
details about quaternions and octonions we refer to \cite{Adler, Baez, Ward}.

The quaternion algebra 
\begin{equation}
\mathbf{H}=\{\sum_{l=0}^{3}a_{l}e_{l}:a_{l}\in 
\mathbb{R}
\}  \label{q1}
\end{equation}%
is a four dimensional non-commutative vector space over $%
\mathbb{R}
$ and the basis satisfy the following multiplication rules:%
\begin{eqnarray}
e_{l}^{2} &=&-1,\text{ }l\in \{1,2,3\};  \notag \\
e_{1}e_{2} &=&-e_{2}e_{1}=e_{3},\text{ }e_{2}e_{3}=-e_{3}e_{2}=e_{1},\text{ }%
e_{3}e_{1}=-e_{1}e_{3}=e_{2}.  \label{q2}
\end{eqnarray}%
($e_{0}$ can be identified with real number $1$). Also, the quaternion
algebra $\mathbf{H}$ is isomorphic to the Clifford algebra $C\ell _{0,2}$.
There are several studies on different types of sequences over quaternion
algebra. For a survey on these researches we refer to \cite{flaut, horadam1,
Iakin, Iakin1, Iyer, Iyer1, halici, halici1, cimen, catarino, szynal,
Ramirez}.

Recently, Tan and et. al. \cite{chaos1, chaos2} introduced a new
generalization of the Fibonacci and Lucas quaternions, named as, the
bi-periodic Fibonacci and Lucas quaternions. They are emerged as a
generalization of the best known quaternions in the literature, such as
classical Fibonacci and Lucas quaternions in \cite{horadam1}, Pell and
Pell-Lucas quaternions in \cite{cimen}, $k-$Fibonacci and $k-$Lucas
quaternions in \cite{Ramirez}.

For $n\geq 0,$ the bi-periodic Fibonacci and Lucas quaternions defined as%
\begin{equation}
Q_{n}=\sum_{l=0}^{3}q_{n+l}e_{l}\text{ \ \ \ and \ \ \ }P_{n}=%
\sum_{l=0}^{3}p_{n+l}e_{l},  \label{1}
\end{equation}%
respectively. Note that $q_{n}$ is the $n$th bi-periodic Fibonacci number
and defined by%
\begin{equation}
q_{n}=\left\{ 
\begin{array}{ll}
aq_{n-1}+q_{n-2}, & \mbox{ if }n\mbox{ is even} \\ 
bq_{n-1}+q_{n-2}, & \mbox{ if }n\mbox{ is odd}%
\end{array}%
\right. ,\text{ }n\geq 2  \label{2}
\end{equation}%
with initial values $q_{0}=0$, $q_{1}=1$ and $a$, $b$ are nonzero numbers
and $p_{n}$ is the $n$th bi-periodic Lucas number and defined by%
\begin{equation}
p_{n}=\left\{ 
\begin{array}{ll}
bp_{n-1}+p_{n-2}, & \mbox{ if }n\mbox{ is even} \\ 
ap_{n-1}+p_{n-2}, & \mbox{ if }n\mbox{ is odd}%
\end{array}%
\right. ,\text{ }n\geq 2  \label{3}
\end{equation}%
with the initial conditions $p_{0}=2$, $p_{1}=a$.

The Binet formula for the bi-periodic Fibonacci quaternion is given by%
\begin{equation}
Q_{n}=\left\{ 
\begin{array}{ll}
\frac{1}{\left( ab\right) ^{\left\lfloor \frac{n}{2}\right\rfloor }}\frac{%
\alpha ^{\ast }\alpha ^{n}-\beta ^{\ast }\beta ^{n}}{\alpha -\beta }, & %
\mbox{ if }n\mbox{ is even} \\ 
\frac{1}{\left( ab\right) ^{\left\lfloor \frac{n}{2}\right\rfloor }}\frac{%
\alpha ^{\ast \ast }\alpha ^{n}-\beta ^{\ast \ast }\beta ^{n}}{\alpha -\beta 
}, & \mbox{ if }n\mbox{ is odd}%
\end{array}%
\right. \text{ \ \ \ \ \ \ \ \ \ \ \ }  \label{4}
\end{equation}%
and the Binet formula for the bi-periodic Lucas quaternion is%
\begin{equation}
P_{n}=\left\{ 
\begin{array}{ll}
\frac{1}{\left( ab\right) ^{\left\lfloor \frac{n+1}{2}\right\rfloor }}\left(
\alpha ^{\ast \ast }\alpha ^{n}+\beta ^{\ast \ast }\beta ^{n}\right) , & %
\mbox{ if }n\mbox{ is even} \\ 
\frac{1}{\left( ab\right) ^{\left\lfloor \frac{n+1}{2}\right\rfloor }}\left(
\alpha ^{\ast }\alpha ^{n}+\beta ^{\ast }\beta ^{n}\right) , & \mbox{ if }n%
\mbox{ is odd}%
\end{array}%
\right.  \label{5}
\end{equation}%
where%
\begin{eqnarray}
\alpha &:&=\frac{ab+\sqrt{a^{2}b^{2}+4ab}}{2},\beta :=\frac{ab-\sqrt{%
a^{2}b^{2}+4ab}}{2}\text{ \ \ \ }  \notag \\
\alpha ^{\ast } &:&=\sum_{l=0}^{3}\frac{a^{\zeta \left( l+1\right) }}{\left(
ab\right) ^{\left\lfloor \frac{l}{2}\right\rfloor }}\alpha ^{l}e_{l},\text{ }%
\beta ^{\ast }:=\sum_{l=0}^{3}\frac{a^{\zeta \left( l+1\right) }}{\left(
ab\right) ^{\left\lfloor \frac{l}{2}\right\rfloor }}\beta ^{l}e_{l}\text{ } 
\notag \\
\alpha ^{\ast \ast } &:&=\sum_{l=0}^{3}\frac{a^{\zeta \left( l\right) }}{%
\left( ab\right) ^{\left\lfloor \frac{l+1}{2}\right\rfloor }}\alpha
^{l}e_{l},\text{ }\beta ^{\ast \ast }:=\sum_{l=0}^{3}\frac{a^{\zeta \left(
l\right) }}{\left( ab\right) ^{\left\lfloor \frac{l+1}{2}\right\rfloor }}%
\beta ^{l}e_{l}.  \label{6}
\end{eqnarray}%
Here $\zeta \left( n\right) =n-2\left\lfloor \frac{n}{2}\right\rfloor $ is
the parity function, i.e., $\zeta \left( n\right) =0$ when $n$ is even and $%
\zeta \left( n\right) =1$ when $n$ is odd. Assume that $a^{2}b^{2}+4ab\neq
0. $ Also we have $\alpha +\beta =ab$, $\alpha -\beta =\sqrt{a^{2}b^{2}+4ab}$
and $\alpha \beta =-ab.$ For the details of the bi-periodic Fibonacci and
Lucas sequences see \cite{Yayenie, Edson, Bilgici, Sahin, Panario}.

The octonion algebra%
\begin{equation}
\mathbf{O}=\{\sum_{l=0}^{7}a_{l}e_{l}:a_{l}\in 
\mathbb{R}
\}  \label{o1}
\end{equation}%
is an eight dimensional non-commutative and non-associative vector space
over $%
\mathbb{R}
,$ and the multiplication rules can be derived from the following table :%
\begin{equation*}
\underset{%
\begin{array}{c}
\\ 
\text{\textit{Table 1: Octonion Multiplication table}}%
\end{array}%
}{%
\begin{tabular}{l|llllllll}
& $1$ & $e_{1}$ & $e_{2}$ & $e_{3}$ & $e_{4}$ & $e_{5}$ & $e_{6}$ & $e_{7}$
\\ \hline
$1$ & $1$ & $e_{1}$ & $e_{2}$ & $e_{3}$ & $e_{4}$ & $e_{5}$ & $e_{6}$ & $%
e_{7}$ \\ 
$e_{1}$ & $e_{1}$ & $-1$ & $e_{3}$ & $-e_{2}$ & $e_{5}$ & $-e_{4}$ & $-e_{7}$
& $e_{6}$ \\ 
$e_{2}$ & $e_{2}$ & $-e_{3}$ & $-1$ & $e_{1}$ & $e_{6}$ & $e_{7}$ & $-e_{4}$
& $-e_{5}$ \\ 
$e_{3}$ & $e_{3}$ & $e_{2}$ & $-e_{1}$ & $-1$ & $e_{7}$ & $-e_{6}$ & $e_{5}$
& $-e_{4}$ \\ 
$e_{4}$ & $e_{4}$ & $-e_{5}$ & $-e_{6}$ & $-e_{7}$ & $-1$ & $e_{1}$ & $e_{2}$
& $e_{3}$ \\ 
$e_{5}$ & $e_{5}$ & $e_{4}$ & $-e_{7}$ & $e_{6}$ & $-e_{1}$ & $-1$ & $-e_{3}$
& $e_{2}$ \\ 
$e_{6}$ & $e_{6}$ & $e_{7}$ & $e_{4}$ & $-e_{5}$ & $-e_{2}$ & $e_{3}$ & $-1$
& $-e_{1}$ \\ 
$e_{7}$ & $e_{7}$ & $-e_{6}$ & $e_{5}$ & $e_{4}$ & $-e_{3}$ & $-e_{2}$ & $%
e_{1}$ & $-1$%
\end{tabular}%
}
\end{equation*}

Motivating by the results in \cite{chaos1, chaos2}, Yilmaz and et. al. \cite%
{yilmaz1, yilmaz2} introduced the bi-periodic Fibonacci and Lucas octonions
as%
\begin{equation}
OQ_{n}=\sum_{l=0}^{7}q_{n+l}e_{l}\text{ \ \ \ \ and\ \ \ \ \ }%
OP_{n}=\sum_{l=0}^{7}p_{n+l}e_{l},  \label{7}
\end{equation}%
respectively.

The Binet formula for the bi-periodic Fibonacci octonion is%
\begin{equation}
OQ_{n}=\left\{ 
\begin{array}{ll}
\frac{1}{\left( ab\right) ^{\left\lfloor \frac{n}{2}\right\rfloor }}\frac{%
\gamma ^{\ast }\alpha ^{n}-\delta ^{\ast }\beta ^{n}}{\alpha -\beta }, & %
\mbox{ if }n\mbox{ is even} \\ 
\frac{1}{\left( ab\right) ^{\left\lfloor \frac{n}{2}\right\rfloor }}\frac{%
\gamma ^{\ast \ast }\alpha ^{n}-\delta ^{\ast \ast }\beta ^{n}}{\alpha
-\beta }, & \mbox{ if }n\mbox{ is odd}%
\end{array}%
\right. \text{ \ \ \ \ \ \ \ \ }  \label{8}
\end{equation}%
and the Binet formula for the bi-periodic Lucas octonion is%
\begin{equation}
OP_{n}=\left\{ 
\begin{array}{ll}
\frac{1}{\left( ab\right) ^{\left\lfloor \frac{n+1}{2}\right\rfloor }}\left(
\gamma ^{\ast \ast }\alpha ^{n}+\delta ^{\ast \ast }\beta ^{n}\right) , & %
\mbox{ if }n\mbox{ is even} \\ 
\frac{1}{\left( ab\right) ^{\left\lfloor \frac{n+1}{2}\right\rfloor }}\left(
\gamma ^{\ast }\alpha ^{n}+\delta ^{\ast }\beta ^{n}\right) , & \mbox{ if }n%
\mbox{ is odd}%
\end{array}%
\right.  \label{9}
\end{equation}%
where%
\begin{eqnarray}
\gamma ^{\ast } &:&=\sum_{l=0}^{7}\frac{a^{\zeta \left( l+1\right) }}{\left(
ab\right) ^{\left\lfloor \frac{l}{2}\right\rfloor }}\alpha ^{l}e_{l},\text{ }%
\delta ^{\ast }:=\sum_{l=0}^{7}\frac{a^{\zeta \left( l+1\right) }}{\left(
ab\right) ^{\left\lfloor \frac{l}{2}\right\rfloor }}\beta ^{l}e_{l}\text{ \
\ \ \ \ \ \ }  \notag \\
\gamma ^{\ast \ast } &:&=\sum_{l=0}^{7}\frac{a^{\zeta \left( l\right) }}{%
\left( ab\right) ^{\left\lfloor \frac{l+1}{2}\right\rfloor }}\alpha
^{l}e_{l},\text{ }\delta ^{\ast \ast }:=\sum_{l=0}^{7}\frac{a^{\zeta \left(
l\right) }}{\left( ab\right) ^{\left\lfloor \frac{l+1}{2}\right\rfloor }}%
\beta ^{l}e_{l}.  \label{10}
\end{eqnarray}%
For related studies on different types of sequences over octonion algebra,
we refer to \cite{oct, oct1, szynal, catarino, halici2}.

In this paper, we present a further generalization of the bi-periodic
Fibonacci quaternions and octonions. We give the generating function, the
Binet formula, and some basic properties of these quaternions and octonions.
This new generalization can be seen as a generalization of the papers which
are given in \cite{chaos1, chaos2, yilmaz1, yilmaz2, catarino}. The results
of this paper not only give a generalization of the bi-periodic Fibonacci
quaternions and octonions, but also include new results such as the matrix
representation and the norm value of the generalized bi-periodic Fibonacci
sequence. The main contribution of this study is one can get a great number
of distinct quaternion and octonion sequences by providing the initial
values in the generalized bi-periodic Fibonacci sequence. To this end, first
consider the generalized bi-periodic Fibonacci sequence, $\left\{
w_{n}\right\} ,$ which is defined in \cite{Edson}\textit{\ }as:%
\begin{equation}
w_{n}=\left\{ 
\begin{array}{ll}
aw_{n-1}+w_{n-2}, & \mbox{ if }n\mbox{ is even} \\ 
bw_{n-1}+w_{n-2}, & \mbox{ if }n\mbox{ is odd}%
\end{array}%
\right. ,n\geq 2  \label{11}
\end{equation}%
with arbitrary initial conditions $w_{0},w_{1}$ where $w_{0},w_{1},a,b$ are
nonzero numbers.\textit{\ }Note that, if we take $w_{0}=0,w_{1}=1$ in $%
\{w_{n}\},$ we get the bi-periodic Fibonacci sequence $\left\{ q_{n}\right\} 
$ in (\ref{2}). If we take $w_{0}=2,w_{1}=b,$ and switch $a$ and $b$ in $%
\{w_{n}\},$ we get the bi-periodic Lucas sequence $\{p_{n}\}$ in (\ref{3}).

In \cite{tan}, the Binet formula of the sequence $\{w_{n}\}$ is given by%
\begin{equation}
w_{n}=\frac{a^{\zeta \left( n+1\right) }}{\left( ab\right) ^{\left\lfloor 
\frac{n}{2}\right\rfloor }}\left( A\alpha ^{n-1}-B\beta ^{n-1}\right) \text{
\ }  \label{13}
\end{equation}%
where%
\begin{equation}
A:=\frac{\alpha w_{1}+bw_{0}}{\alpha -\beta }\text{ and }B:=\frac{\beta
w_{1}+bw_{0}}{\alpha -\beta }.  \label{14}
\end{equation}%
For more results related to the sequence $\{w_{n}\},$ we refer to \cite{tan}.

\section{The generalized bi-periodic Fibonacci quaternions}

In this section, we introduce the generalized bi-periodic Fibonacci
quaternions and give some basic properties of them. These results can be
seen as a generalization of the results in \cite{chaos1, chaos2, catarino}.

\begin{definition}
The generalized bi-periodic Fibonacci quaternions $\{W_{n}\}$ are defined by%
\begin{equation}
W_{n}=\sum_{l=0}^{3}w_{n+l}e_{l},  \label{15}
\end{equation}%
where $w_{n}$ is defined in (\ref{11}).
\end{definition}

In the following, we give several number of different sequences which are
special cases of $\{W_{n}\}:$

\begin{enumerate}
\item If we take the initial conditions $w_{0}=0$ and $w_{1}=1,$ we get the
bi-periodic Fibonacci quaternions in \cite{chaos1}.

\item If we take the initial conditions $w_{0}=2$ and $w_{1}=b,$ we get the
bi-periodic Lucas quaternions in \cite{chaos2}. (Note that we switch $a$ and 
$b$).

\item If we take the initial conditions $w_{0}=w_{1}=1$ and $a=b=2$ in $%
\{w_{n}\}$, we get the modified Pell quaternion numbers in \cite{catarino}.

\item If we take $a=b=1$ in $\{w_{n}\},$ we get the Horadam quaternion
numbers in \cite{halici1} with the case of $q=1$.
\end{enumerate}

\begin{theorem}
The generating function for the generalized bi-periodic Fibonacci
quaternions $W_{n}$ is%
\begin{equation}
G\left( t\right) =\frac{W_{0}+\left( W_{1}-bW_{0}\right) t+\left( a-b\right)
\sum_{s=0}^{3}R\left( t,s\right) e_{s}}{1-bt-t^{2}}  \label{16}
\end{equation}%
where%
\begin{equation}
R\left( t,s\right) :=\left( f\left( t\right) -\sum_{k=1}^{\left\lfloor \frac{%
s+1}{2}\right\rfloor }w_{2k-1}t^{2k-1}\right) t^{1-s},\text{ \ \ \ \ \ \ \ \
\ \ \ \ }  \label{17}
\end{equation}%
\ 
\begin{equation}
f\left( t\right) :=\sum_{n=1}^{\infty }w_{2n-1}t^{2n-1}=\frac{w_{1}t+\left(
bw_{0}-w_{1}\right) t^{3}}{1-\left( ab+2\right) t^{2}+t^{4}}.\text{ \ \ \ \
\ \ \ \ }  \label{18}
\end{equation}
\end{theorem}

\begin{proof}
By using the similar method in \cite[Theorem 1]{chaos1} and considering the
relation%
\begin{equation*}
w_{2n-1}=\left( ab+2\right) w_{2n-3}-w_{2n-5},
\end{equation*}%
we get the desired result.
\end{proof}

In the following theorem, we state the Binet formula for the generalized
bi-periodic Fibonacci quaternions and so derive some well-known mathematical
properties such as Catalan's like identity and Cassini's like identity.

\begin{theorem}
The Binet formula for the generalized bi-periodic Fibonacci quaternion is%
\begin{equation}
W_{n}=\left\{ 
\begin{array}{ll}
\frac{1}{\left( ab\right) ^{\left\lfloor \frac{n}{2}\right\rfloor }}\left(
A\alpha ^{\ast }\alpha ^{n-1}-B\beta ^{\ast }\beta ^{n-1}\right) , & 
\mbox{
if }n\mbox{ is even} \\ 
\frac{1}{\left( ab\right) ^{\left\lfloor \frac{n}{2}\right\rfloor }}\left(
A\alpha ^{\ast \ast }\alpha ^{n-1}-B\beta ^{\ast \ast }\beta ^{n-1}\right) ,
& \mbox{ if }n\mbox{ is odd}%
\end{array}%
\right.  \label{19}
\end{equation}%
where $A,B,\alpha ^{\ast },\beta ^{\ast },\alpha ^{\ast \ast },$ and $\beta
^{\ast \ast }$ defined in (\ref{6}) and (\ref{14}).
\end{theorem}

\begin{proof}
By using the definition of the sequence $\{w_{n}\}$ and the Binet formula in
(\ref{13}), we can easily obtained the desired result.
\end{proof}

By using the Binet formula for the generalized bi-periodic Fibonacci
quaternion sequences, we obtain the following identity.

\begin{theorem}
\textit{(Catalan's like identity)} For nonnegative integer number $n$ and
even integer $r,$ such that $r\leq n,$ we have%
\begin{eqnarray}
&&W_{n-r}W_{n+r}-W_{n}^{2}  \notag \\
&=&\left\{ 
\begin{array}{ll}
\frac{AB\left( \alpha ^{r}-\beta ^{r}\right) }{\left( \alpha \beta \right)
^{r+1}}\left[ \alpha ^{\ast }\beta ^{\ast }\beta ^{r}-\beta ^{\ast }\alpha
^{\ast }\alpha ^{r}\right] , & \mbox{
if }n\mbox{ is even} \\ 
\frac{AB\left( \alpha ^{r}-\beta ^{r}\right) }{\left( \alpha \beta \right)
^{r}}\left[ \alpha ^{\ast \ast }\beta ^{\ast \ast }\beta ^{r}-\beta ^{\ast
\ast }\alpha ^{\ast \ast }\alpha ^{r}\right] , & \mbox{ if }n\mbox{ is odd}%
\end{array}%
\right. .  \label{20}
\end{eqnarray}
\end{theorem}

\begin{proof}
For even $n$, we have%
\begin{equation*}
W_{n-r}W_{n+r}-W_{n}^{2}\text{ \ \ \ \ \ \ \ \ \ \ \ \ \ \ \ \ \ \ \ \ \ \ \
\ \ \ \ \ \ \ \ \ \ \ \ \ \ \ \ \ \ \ \ \ \ \ \ \ \ \ \ \ \ \ \ \ \ \ \ \ \
\ \ \ \ \ \ \ \ \ \ \ \ \ \ \ \ \ \ \ \ \ \ \ }
\end{equation*}%
\begin{equation*}
=\frac{1}{\left( ab\right) ^{n}}\left( A\alpha ^{\ast }\alpha
^{n-r-1}-B\beta ^{\ast }\beta ^{n-r-1}\right) \left( A\alpha ^{\ast }\alpha
^{n+r-1}-B\beta ^{\ast }\beta ^{n+r-1}\right) \text{ \ \ \ \ \ \ \ \ \ \ \ \
\ \ \ \ \ \ \ \ \ \ \ \ \ \ \ \ \ \ }
\end{equation*}%
\begin{equation*}
-\frac{1}{\left( ab\right) ^{n}}\left( A\alpha ^{\ast }\alpha ^{n-1}-B\beta
^{\ast }\beta ^{n-1}\right) \left( A\alpha ^{\ast }\alpha ^{n-1}-B\beta
^{\ast }\beta ^{n-1}\right) \text{ \ \ \ \ \ \ \ \ \ \ \ \ \ \ \ \ \ \ \ \ \
\ }
\end{equation*}%
\begin{equation*}
=\frac{1}{\left( ab\right) ^{n}}\left[ AB\alpha ^{\ast }\beta ^{\ast }\left(
\alpha \beta \right) ^{n-1}\left( 1-\frac{\beta ^{r}}{\alpha ^{r}}\right)
+BA\beta ^{\ast }\alpha ^{\ast }\left( \alpha \beta \right) ^{n-1}\left( 1-%
\frac{\alpha ^{r}}{\beta ^{r}}\right) \right] \text{ \ \ \ \ \ \ \ \ \ \ \ \
\ \ \ \ \ \ \ \ \ }
\end{equation*}%
\begin{equation*}
=\frac{\left( \alpha \beta \right) ^{n-1}AB}{\left( ab\right) ^{n}}\left[
\alpha ^{\ast }\beta ^{\ast }\left( 1-\frac{\beta ^{r}}{\alpha ^{r}}\right)
+\beta ^{\ast }\alpha ^{\ast }\left( 1-\frac{\alpha ^{r}}{\beta ^{r}}\right) %
\right] \text{ \ \ \ \ \ \ \ \ \ \ \ \ \ \ \ \ \ \ \ \ \ \ \ \ \ \ \ \ \ \ \
\ \ \ \ \ \ \ \ \ \ \ \ \ \ \ \ \ \ \ }
\end{equation*}%
\begin{equation*}
=\frac{AB}{\left( -1\right) ^{n}\alpha \beta }\left[ \alpha ^{\ast }\beta
^{\ast }\left( \frac{\alpha ^{r}-\beta ^{r}}{\alpha ^{r}}\right) +\beta
^{\ast }\alpha ^{\ast }\left( \frac{\beta ^{r}-\alpha ^{r}}{\beta ^{r}}%
\right) \right] \text{ \ \ \ \ \ \ \ \ \ \ \ \ \ \ \ \ \ \ \ \ \ \ \ \ \ \ \
\ \ \ \ \ \ \ \ \ \ \ \ \ \ \ \ \ \ }
\end{equation*}%
\begin{equation*}
=\frac{AB}{\left( -1\right) ^{n}\left( \alpha \beta \right) ^{r+1}}\left[
\alpha ^{\ast }\beta ^{\ast }\beta ^{r}\left( \alpha ^{r}-\beta ^{r}\right)
+\beta ^{\ast }\alpha ^{\ast }\alpha ^{r}\left( \beta ^{r}-\alpha
^{r}\right) \right] \text{ \ \ \ \ \ \ \ \ \ \ \ \ \ \ \ \ \ \ \ \ \ \ \ \ \
\ \ \ \ \ \ \ \ \ \ \ \ \ \ \ \ \ \ \ \ \ \ \ \ \ \ \ \ \ \ \ \ \ \ }
\end{equation*}%
\begin{equation*}
=\frac{AB\left( \alpha ^{r}-\beta ^{r}\right) }{\left( \alpha \beta \right)
^{r+1}}\left[ \alpha ^{\ast }\beta ^{\ast }\beta ^{r}-\beta ^{\ast }\alpha
^{\ast }\alpha ^{r}\right] .\text{ \ \ \ \ \ \ \ \ \ \ \ \ \ \ \ \ \ \ \ \ \
\ \ \ \ \ \ \ \ \ \ \ \ \ \ \ \ \ \ \ \ \ \ \ \ \ \ \ \ \ \ \ }
\end{equation*}%
Similarly, it can be proven for odd $n.$
\end{proof}

If we take the initial conditions $w_{0}=0$ and $w_{1}=1$ in (\ref{20}), we
get the result in \cite[Theorem 5]{chaos1}, and if we take the initial
conditions $w_{0}=2$ and $w_{1}=b$ in (\ref{20}), we get the result in \cite[%
Theorem 5]{chaos2}. Also, it is clear that if we take $r=2$ in the above
theorem we obtain the following result.

\begin{corollary}
(Cassini's like identity) For nonnegative even integer number $n$, we have%
\begin{equation}
W_{n-2}W_{n+2}-W_{n}^{2}=\frac{AB\left( \alpha ^{2}-\beta ^{2}\right) }{%
\left( \alpha \beta \right) ^{3}}\left[ \alpha ^{\ast }\beta ^{\ast }\beta
^{2}-\beta ^{\ast }\alpha ^{\ast }\alpha ^{2}\right] .  \label{21}
\end{equation}
\end{corollary}

If we take the initial conditions $w_{0}=0$ and $w_{1}=1$ in (\ref{21}), we
get the result in \cite[Theorem 3]{chaos1}, and if we take the initial
conditions $w_{0}=2$ and $w_{1}=b$ in (\ref{21}), we get the result in \cite[%
Theorem 3]{chaos2}.

To present the Cassini's like identity as a different manner, now we give a
matrix representation for the even indices terms of the generalized
bi-periodic Fibonacci quaternions.

\begin{theorem}
For $n\geq 1,$ we have%
\begin{equation}
\left[ 
\begin{array}{cc}
W_{2n} & W_{2\left( n-1\right) } \\ 
W_{2\left( n+1\right) } & W_{2n}%
\end{array}%
\right] =\left[ 
\begin{array}{cc}
W_{2} & W_{0} \\ 
W_{4} & W_{2}%
\end{array}%
\right] \left[ 
\begin{array}{cc}
ab+2 & 1 \\ 
-1 & 0%
\end{array}%
\right] ^{n-1}.  \label{22}
\end{equation}
\end{theorem}

\begin{proof}
We prove it by using induction on $n$. It is clear that the result is true
when $n=1$. Assume that it is true for any integer $m$ such that $1\leq
m\leq n$. Then by using induction assumption, we have%
\begin{equation*}
\left[ 
\begin{array}{cc}
W_{2} & W_{0} \\ 
W_{4} & W_{2}%
\end{array}%
\right] \left[ 
\begin{array}{cc}
ab+2 & 1 \\ 
-1 & 0%
\end{array}%
\right] ^{n}\text{\ \ \ \ \ \ \ \ \ \ \ \ \ \ \ \ \ \ \ \ \ \ \ \ \ \ \ \ \
\ \ \ \ \ \ \ \ \ \ \ \ \ \ \ \ \ \ \ \ \ \ \ \ \ \ }
\end{equation*}%
\begin{equation*}
=\left[ 
\begin{array}{cc}
W_{2} & W_{0} \\ 
W_{4} & W_{2}%
\end{array}%
\right] \left[ 
\begin{array}{cc}
ab+2 & 1 \\ 
-1 & 0%
\end{array}%
\right] ^{n-1}\left[ 
\begin{array}{cc}
ab+2 & 1 \\ 
-1 & 0%
\end{array}%
\right] \text{\ \ \ \ \ \ \ \ \ \ \ \ \ \ \ \ \ \ \ \ \ \ \ \ \ \ \ \ \ }
\end{equation*}%
\begin{equation*}
=\left[ 
\begin{array}{cc}
W_{2n} & W_{2\left( n-1\right) } \\ 
W_{2\left( n+1\right) } & W_{2n}%
\end{array}%
\right] \left[ 
\begin{array}{cc}
ab+2 & 1 \\ 
-1 & 0%
\end{array}%
\right] \text{ \ \ \ \ \ \ \ \ \ \ \ \ \ \ \ \ \ \ \ \ \ \ \ \ \ \ \ \ \ \ \
\ \ \ \ \ \ \ \ }
\end{equation*}%
\begin{equation*}
=\left[ 
\begin{array}{cc}
\left( ab+2\right) W_{2n}-W_{2\left( n-1\right) } & W_{2n} \\ 
\left( ab+2\right) W_{2\left( n+1\right) }-W_{2n} & W_{2\left( n+1\right) }%
\end{array}%
\right] =\left[ 
\begin{array}{cc}
W_{2\left( n+1\right) } & W_{2n} \\ 
W_{2\left( n+2\right) } & W_{2\left( n+1\right) }%
\end{array}%
\right] \text{ \ \ \ \ }
\end{equation*}%
which completes the proof.
\end{proof}

\begin{corollary}
For $n\geq 1,$ we have%
\begin{equation}
W_{2\left( n-1\right) }W_{2\left( n+1\right)
}-W_{2n}^{2}=W_{0}W_{4}-W_{2}^{2}.  \label{23}
\end{equation}
\end{corollary}

By means of this formula we can state the Cassini's like identity for the
bi-periodic Fibonacci quaternions as:%
\begin{equation*}
Q_{2\left( n-1\right) }Q_{2\left( n+1\right)
}-Q_{2n}^{2}=Q_{0}Q_{4}-Q_{2}^{2},
\end{equation*}%
which is not given before.

Following result gives the relation between the generalized bi-periodic
Fibonacci quaternions $\left\{ W_{n}\right\} $ and the bi-periodic Fibonacci
quaternions $\left\{ Q_{n}\right\} $.

\begin{theorem}
For any natural number $n$, we have%
\begin{equation}
W_{2\left( n+1\right) }Q_{2n}-W_{2n}Q_{2\left( n+1\right) }=\frac{1}{ab}%
\left[ A\alpha ^{\ast }\beta ^{\ast }\beta -B\beta ^{\ast }\alpha ^{\ast
}\alpha \right] .  \label{24}
\end{equation}
\end{theorem}

\begin{proof}
By using the Binet formula for the generalized bi-periodic Fibonacci
quaternions and the bi-periodic Fibonacci quaternions, we can easily obtain
the desired result.
\end{proof}

The norm value of the generalized bi-periodic Fibonacci quaternions is
defined by 
\begin{equation*}
Nr\left( W_{n}\right) :=W_{n}\overline{W_{n}},
\end{equation*}%
where $\overline{W_{n}}$ $:=w_{n}e_{0}-w_{1}e_{1}-w_{2}e_{2}-w_{3}e_{3}$ is
the conjugate of the generalized bi-periodic Fibonacci quaternion. Thus we
have%
\begin{equation*}
Nr\left( W_{n}\right) =w_{n}^{2}+w_{n+1}^{2}+w_{n+2}^{2}+w_{n+3}^{2}.
\end{equation*}%
By using the definition of the norm value and the Binet formula of the
sequence $\left\{ w_{n}\right\} $, then by making some necessary
calculations, we obtain the following result.

\begin{theorem}
The norm value of the generalized bi-periodic Fibonacci quaternions can be
stated as%
\begin{equation}
Nr\left( W_{n}\right) =T\left( n\right) +T\left( n+1\right) ,  \label{n1}
\end{equation}%
where%
\begin{equation}
T\left( n\right) :=\frac{a^{2\zeta \left( n+1\right) }}{\left( ab\right)
^{n-\zeta \left( n\right) }\left( \alpha \beta \right) ^{2}\left( \alpha
-\beta \right) ^{2}}\left[ w_{1}^{2}X+2w_{0}w_{1}bY+w_{0}^{2}b^{2}Z\right] ,
\label{n2}
\end{equation}%
and%
\begin{eqnarray}
&&X:=\alpha ^{2n}\left( \alpha ^{4}+\left( \alpha \beta \right) ^{2}\right)
+\beta ^{2n}\left( \beta ^{4}+\left( \alpha \beta \right) ^{2}\right)
-4\left( \alpha \beta \right) ^{n+2},  \notag \\
&&Y:=\alpha ^{2n-1}\left( \alpha ^{4}+\left( \alpha \beta \right)
^{2}\right) +\beta ^{2n-1}\left( \beta ^{4}+\left( \alpha \beta \right)
^{2}\right) +2\left( \alpha \beta \right) ^{n+2},  \notag \\
&&Z:=\alpha ^{2n-2}\left( \alpha ^{4}+\left( \alpha \beta \right)
^{2}\right) +\beta ^{2n-2}\left( \beta ^{4}+\left( \alpha \beta \right)
^{2}\right) -4\left( \alpha \beta \right) ^{n+1}.  \label{n3}
\end{eqnarray}
\end{theorem}

Note that, if we take the initial conditions $w_{0}=0,w_{1}=1$ and $a=b=2,$
we get the norm value of the Pell quaternions in \cite[Theorem 3.1]{szynal},
and if we take $a=b$ in $\{w_{n}\},$ we get the norm value of the Horadam
quaternion numbers in \cite{halici1} with the case of $q=1$.

Finally, we give some summation formulas for the generalized bi-periodic
Fibonacci quaternions.

\begin{theorem}
For $n\geq 1,$ we have%
\begin{eqnarray}
\text{ \ \ \ \ \ }\left( i\right) \sum_{r=0}^{n-1}W_{r} &=&\frac{%
W_{n}-W_{n-2}+W_{n+1}-W_{n-1}}{ab}  \notag \\
&&-\frac{A\alpha ^{\ast }\beta ^{2}-B\beta ^{\ast }\alpha ^{2}-ab\left(
A\alpha ^{\ast \ast }\beta -B\beta ^{\ast \ast }\alpha \right) }{\left(
ab\right) ^{2}},  \label{25}
\end{eqnarray}%
\begin{equation}
\left( ii\right) \sum_{r=0}^{n-1}W_{2r}=\frac{W_{2n}-W_{2n-2}}{ab}-\frac{%
A\alpha ^{\ast }\beta ^{2}-B\beta ^{\ast }\alpha ^{2}}{\left( ab\right) ^{2}}%
,\text{ \ \ \ \ \ }  \label{26}
\end{equation}%
\begin{equation}
\left( iii\right) \sum_{r=0}^{n-1}W_{2r+1}=\frac{W_{2n+1}-W_{2n-1}}{ab}+%
\frac{A\alpha ^{\ast \ast }\beta -B\beta ^{\ast \ast }\alpha }{ab}.
\label{27}
\end{equation}
\end{theorem}

\begin{proof}
$\left( i\right) $If $n$ is odd,%
\begin{equation*}
\sum_{r=0}^{n-1}W_{r}=\sum_{r=0}^{\frac{n-1}{2}}W_{2r}+\sum_{r=0}^{\frac{n-3%
}{2}}W_{2r+1}\text{ \ \ \ \ \ \ \ \ \ \ \ \ \ \ \ \ \ \ \ \ \ \ \ \ \ \ \ \
\ \ \ \ \ \ \ \ \ \ \ \ \ \ \ \ \ \ \ \ \ \ \ \ \ \ \ \ \ \ \ \ \ \ \ \ \ \ }
\end{equation*}%
\begin{equation*}
=\sum_{r=0}^{\frac{n-1}{2}}\frac{1}{\left( ab\right) ^{r}}\left( A\alpha
^{\ast }\alpha ^{2r-1}-B\beta ^{\ast }\beta ^{2r-1}\right) +\sum_{r=0}^{%
\frac{n-3}{2}}\frac{1}{\left( ab\right) ^{r}}\left( A\alpha ^{\ast \ast
}\alpha ^{2r}-B\beta ^{\ast \ast }\beta ^{2r}\right) \text{ \ \ \ \ \ \ \ \
\ \ \ \ \ \ \ \ }
\end{equation*}%
\begin{equation*}
=A\alpha ^{\ast }\alpha ^{-1}\sum_{r=0}^{\frac{n-1}{2}}\left( \frac{\alpha
^{2}}{ab}\right) ^{r}-B\beta ^{\ast }\beta ^{-1}\sum_{r=0}^{\frac{n-1}{2}%
}\left( \frac{\beta ^{2}}{ab}\right) ^{r}\text{ \ \ \ \ \ \ \ \ \ \ \ \ \ \
\ \ \ \ \ \ \ \ \ \ \ \ \ \ \ \ \ \ \ \ \ \ \ \ \ \ \ \ \ \ \ \ \ \ \ \ \ \
\ \ \ \ \ \ }
\end{equation*}%
\begin{equation*}
+A\alpha ^{\ast \ast }\sum_{r=0}^{\frac{n-3}{2}}\left( \frac{\alpha ^{2}}{ab}%
\right) ^{r}-B\beta ^{\ast \ast }\sum_{r=0}^{\frac{n-3}{2}}\left( \frac{%
\beta ^{2}}{ab}\right) ^{r}\text{ \ \ \ \ \ \ \ \ \ \ \ \ \ \ \ \ \ \ \ \ \
\ \ \ \ \ \ \ \ \ \ \ \ \ \ \ \ \ \ \ \ \ \ \ \ \ \ \ \ \ \ \ \ \ \ \ }
\end{equation*}%
\begin{equation*}
=A\alpha ^{\ast }\alpha ^{-1}\frac{\left( \frac{\alpha ^{2}}{ab}\right) ^{%
\frac{n-1}{2}+1}-1}{\frac{\alpha ^{2}}{ab}-1}-B\beta ^{\ast }\beta ^{-1}%
\frac{\left( \frac{\beta ^{2}}{ab}\right) ^{\frac{n-1}{2}+1}-1}{\frac{\beta
^{2}}{ab}-1}\text{ \ \ \ \ \ \ \ \ \ \ \ \ \ \ \ \ \ \ \ \ \ \ \ \ \ \ \ \ \
\ \ \ \ \ \ \ \ \ \ \ \ \ \ \ }
\end{equation*}%
\begin{equation*}
+A\alpha ^{\ast \ast }\frac{\left( \frac{\alpha ^{2}}{ab}\right) ^{\frac{n-3%
}{2}+1}-1}{\frac{\alpha ^{2}}{ab}-1}-B\beta ^{\ast \ast }\frac{\left( \frac{%
\beta ^{2}}{ab}\right) ^{\frac{n-3}{2}+1}-1}{\frac{\beta ^{2}}{ab}-1}\text{
\ \ \ \ \ \ \ \ \ \ \ \ \ \ \ \ \ \ \ \ \ \ \ \ \ \ \ \ \ \ \ \ \ \ \ \ \ \
\ \ \ }
\end{equation*}%
\begin{equation*}
=A\alpha ^{\ast }\alpha ^{-1}\frac{\alpha ^{n+1}-\left( ab\right) ^{\frac{n+1%
}{2}}}{\left( \alpha ^{2}-ab\right) \left( ab\right) ^{\frac{n-1}{2}}}%
-B\beta ^{\ast }\beta ^{-1}\frac{\beta ^{n+1}-\left( ab\right) ^{\frac{n+1}{2%
}}}{\left( \beta ^{2}-ab\right) \left( ab\right) ^{\frac{n-1}{2}}}\text{ \ \
\ \ \ \ \ \ \ \ \ \ \ \ \ \ \ \ \ \ \ \ \ \ \ \ \ \ \ \ \ \ \ \ \ \ \ }
\end{equation*}%
\begin{equation*}
+A\alpha ^{\ast \ast }\frac{\alpha ^{n-1}-\left( ab\right) ^{\frac{n-1}{2}}}{%
\left( \alpha ^{2}-ab\right) \left( ab\right) ^{\frac{n-3}{2}}}-B\beta
^{\ast \ast }\frac{\beta ^{n-1}-\left( ab\right) ^{\frac{n-1}{2}}}{\left(
\beta ^{2}-ab\right) \left( ab\right) ^{\frac{n-3}{2}}}\text{ \ \ \ \ \ \ \
\ \ \ \ \ \ \ \ \ \ \ \ \ \ \ \ \ \ \ \ \ \ \ \ \ \ \ \ }
\end{equation*}%
\begin{equation*}
=A\alpha ^{\ast }\frac{\alpha ^{n+1}-\left( ab\right) ^{\frac{n+1}{2}}}{%
\alpha ^{2}\left( ab\right) ^{\frac{n+1}{2}}}-B\beta ^{\ast }\frac{\beta
^{n+1}-\left( ab\right) ^{\frac{n+1}{2}}}{\beta ^{2}\left( ab\right) ^{\frac{%
n+1}{2}}}\text{ \ \ \ \ \ \ \ \ \ \ \ \ \ \ \ \ \ \ \ \ \ \ \ \ \ \ \ \ \ \
\ \ \ \ \ \ \ \ \ \ \ \ \ \ \ \ \ \ \ \ \ \ \ \ \ \ \ }
\end{equation*}%
\begin{equation*}
+A\alpha ^{\ast \ast }\frac{\alpha ^{n-1}-\left( ab\right) ^{\frac{n-1}{2}}}{%
\alpha \left( ab\right) ^{\frac{n-1}{2}}}-B\beta ^{\ast \ast }\frac{\beta
^{n-1}-\left( ab\right) ^{\frac{n-1}{2}}}{\beta \left( ab\right) ^{\frac{n-1%
}{2}}}\text{ \ \ \ \ \ \ \ \ \ \ \ \ \ \ \ \ \ \ \ \ \ \ \ \ \ \ \ \ \ \ \ \
\ \ \ \ \ \ \ \ \ \ \ \ }
\end{equation*}%
\begin{equation*}
=\frac{1}{\left( ab\right) ^{\frac{n+1}{2}}}\left( A\alpha ^{\ast }\frac{%
\alpha ^{n+1}-\left( ab\right) ^{\frac{n+1}{2}}}{\alpha ^{2}}-B\beta ^{\ast }%
\frac{\beta ^{n+1}-\left( ab\right) ^{\frac{n+1}{2}}}{\beta ^{2}}\right) 
\text{ \ \ \ \ \ \ \ \ \ \ \ \ \ \ \ \ \ \ \ \ \ \ \ \ \ \ \ \ \ \ \ \ \ \ \
\ }
\end{equation*}%
\begin{equation*}
+\frac{1}{\left( ab\right) ^{\frac{n-1}{2}}}\left( A\alpha ^{\ast \ast }%
\frac{\alpha ^{n-1}-\left( ab\right) ^{\frac{n-1}{2}}}{\alpha }-B\beta
^{\ast \ast }\frac{\beta ^{n-1}-\left( ab\right) ^{\frac{n-1}{2}}}{\beta }%
\right) \text{ \ \ \ \ \ \ \ \ \ \ \ \ \ \ \ \ \ \ \ \ \ \ \ }
\end{equation*}%
\begin{equation*}
=\frac{1}{\left( ab\right) ^{\frac{n+1}{2}}}\left( A\alpha ^{\ast }\alpha
^{n-1}-B\beta ^{\ast }\beta ^{n-1}-\frac{A\alpha ^{\ast }\left( ab\right) ^{%
\frac{n+1}{2}}}{\alpha ^{2}}-\frac{B\beta ^{\ast }\left( ab\right) ^{\frac{%
n+1}{2}}}{\beta ^{2}}\right) \text{ \ \ \ \ \ \ \ \ \ \ \ \ \ \ \ \ \ \ \ \
\ \ }
\end{equation*}%
\begin{equation*}
+\frac{1}{\left( ab\right) ^{\frac{n-1}{2}}}\left( A\alpha ^{\ast \ast
}\alpha ^{n-2}-B\beta ^{\ast \ast }\beta ^{n-2}-\frac{A\alpha ^{\ast \ast
}\left( ab\right) ^{\frac{n-1}{2}}}{\alpha }+\frac{B\beta ^{\ast \ast
}\left( ab\right) ^{\frac{n-1}{2}}}{\beta }\right) \text{ \ \ \ \ \ \ }
\end{equation*}%
\begin{equation*}
=\frac{1}{\left( ab\right) ^{\frac{n+1}{2}}}\left( A\alpha ^{\ast }\alpha
^{n-1}-B\beta ^{\ast }\beta ^{n-1}\right) -\left( \frac{A\alpha ^{\ast }}{%
\alpha ^{2}}-\frac{B\beta ^{\ast }}{\beta ^{2}}\right) \text{ \ \ \ \ \ \ \
\ \ \ \ \ \ \ \ \ \ \ \ \ \ \ \ \ \ \ \ \ \ \ \ \ \ \ \ \ \ \ \ \ \ \ \ \ }
\end{equation*}%
\begin{equation*}
+\frac{1}{\left( ab\right) ^{\frac{n-1}{2}}}\left( A\alpha ^{\ast \ast
}\alpha ^{n-2}-B\beta ^{\ast \ast }\beta ^{n-2}\right) -\left( \frac{A\alpha
^{\ast \ast }}{\alpha }-\frac{B\beta ^{\ast \ast }}{\beta }\right) \text{ \
\ \ \ \ \ \ \ \ \ \ \ \ \ \ \ \ \ \ \ \ \ \ \ \ \ \ \ }
\end{equation*}%
\begin{equation*}
=\frac{W_{n+1}-W_{n-2}+W_{n}-W_{n-1}}{ab}-\frac{A\alpha ^{\ast }\beta
^{2}-B\beta ^{\ast }\alpha ^{2}-ab\left( A\alpha ^{\ast \ast }\beta -B\beta
^{\ast \ast }\alpha \right) }{\left( ab\right) ^{2}}.\text{\ \ \ \ \ \ \ \ \
\ }
\end{equation*}

Similarly, it can be proven for even $n$. Also, the same procedure can be
applied for $\left( ii\right) $ and\ $\left( iii\right) .$
\end{proof}

Note that, if we take the initial conditions $w_{0}=0$ and $w_{1}=1$ in the
above theorem, we get the summation formulas for the bi-periodic Fibonacci
quaternions, and by taking the initial conditions $w_{0}=2$ and $w_{1}=b,$
we get the summation formulas for the bi-periodic Lucas quaternions which
are not given before.

\section{The generalized bi-periodic Fibonacci octonions}

In this section, we introduce the generalized bi-periodic Fibonacci
octonions and give some basic properties of them. These results can be seen
as a generalization of the papers in \cite{yilmaz1} and \cite{yilmaz2}. Most
of the results can be obtained analogously to the results for the
generalized bi-periodic Fibonacci quaternions, so we omit some proofs.

\begin{definition}
The generalized bi-periodic Fibonacci octonions $\{OW_{n}\}$ are defined by%
\begin{equation}
OW_{n}=\sum_{l=0}^{7}w_{n+l}e_{l},  \label{28}
\end{equation}%
where $w_{n}$ is defined in (\ref{13}).
\end{definition}

Note that, if we take the initial conditions $w_{0}=0$ and $w_{1}=1,$ we get
the bi-periodic Fibonacci octonions in \cite{yilmaz1}. If we take the
initial conditions $w_{0}=2$ and $w_{1}=b,$ we get the bi-periodic Lucas
octonions in \cite{yilmaz2}. Also, if we take $a=b=1$ in $\{w_{n}\},$ we get
the Horadam octonion numbers in \cite{halici2} with the case of $q=1$.

\begin{theorem}
The generating function for the generalized bi-periodic Fibonacci octonion $%
OW_{n}$ is%
\begin{equation}
G^{^{\prime }}\left( t\right) =\frac{OW_{0}+\left( OW_{1}-bOW_{0}\right)
t+\left( a-b\right) \sum_{s=0}^{7}R^{^{\prime }}\left( t,s\right) e_{s}}{%
1-bt-t^{2}}  \label{29}
\end{equation}%
where%
\begin{equation}
R^{^{\prime }}\left( t,s\right) :=\left( f\left( t\right)
-\sum_{k=1}^{\left\lfloor \frac{s+1}{2}\right\rfloor
}w_{2k-1}t^{2k-1}\right) t^{1-s}\text{ \ \ \ \ \ \ \ \ \ \ \ \ \ \ \ \ \ \ \
\ }  \label{30}
\end{equation}%
and $f\left( t\right) $ is defined in (\ref{18}).
\end{theorem}

Note that, if we take the initial conditions $w_{0}=0$ and $w_{1}=1,$ we get
the generating function of the bi-periodic Fibonacci octonions in \cite[%
Theorem 2.4]{yilmaz1}.

\begin{theorem}
The Binet formula for the generalized bi-periodic Fibonacci octonion is%
\begin{equation}
OW_{n}=\left\{ 
\begin{array}{ll}
\frac{1}{\left( ab\right) ^{\left\lfloor \frac{n}{2}\right\rfloor }}\left(
A\gamma ^{\ast }\alpha ^{n-1}-B\delta ^{\ast }\beta ^{n-1}\right) , & 
\mbox{
if }n\mbox{ is even} \\ 
\frac{1}{\left( ab\right) ^{\left\lfloor \frac{n}{2}\right\rfloor }}\left(
A\gamma ^{\ast \ast }\alpha ^{n-1}-B\delta ^{\ast \ast }\beta ^{n-1}\right) ,
& \mbox{ if }n\mbox{ is odd}%
\end{array}%
\right.  \label{31}
\end{equation}%
where $A,B,\gamma ^{\ast },\delta ^{\ast },\gamma ^{\ast \ast },$ and $%
\delta ^{\ast \ast }$ defined in (\ref{10})and (\ref{14}).
\end{theorem}

\begin{theorem}
\textit{(Catalan's like identity)} For nonnegative integer number $n$ and
odd integer $r,$ such that $r\leq n,$ we have%
\begin{eqnarray}
&&OW_{n-r}OW_{n+r}-OW_{n}^{2}  \notag \\
&=&\left\{ 
\begin{array}{ll}
\frac{AB\left( \alpha ^{r}-\beta ^{r}\right) }{\left( \alpha \beta \right)
^{r+1}}\left[ \gamma ^{\ast }\delta ^{\ast }\beta ^{r}-\delta ^{\ast }\gamma
^{\ast }\alpha ^{r}\right] , & \mbox{
if }n\mbox{ is even} \\ 
\frac{AB\left( \alpha ^{r}-\beta ^{r}\right) }{\left( \alpha \beta \right)
^{r}}\left[ \gamma ^{\ast \ast }\delta ^{\ast \ast }\beta ^{r}-\delta ^{\ast
\ast }\gamma ^{\ast \ast }\alpha ^{r}\right] , & \mbox{ if }n\mbox{ is odd}%
\end{array}%
\right. .  \label{32}
\end{eqnarray}
\end{theorem}

It is clear that, if we take $r=2$ in the above theorem we obtain the
Cassini's like identity.

\begin{theorem}
For any natural number $n$, we have%
\begin{equation}
OW_{2\left( n+1\right) }OQ_{2n}-OW_{2n}OQ_{2\left( n+1\right) }=\frac{1}{ab}%
\left[ A\gamma ^{\ast }\delta ^{\ast }\beta -B\delta ^{\ast }\gamma ^{\ast
}\alpha \right] .  \label{33}
\end{equation}
\end{theorem}

\begin{theorem}
For the generalized bi-periodic Fibonacci octonions, we have%
\begin{eqnarray}
\left( i\right) \sum_{r=0}^{n-1}OW_{r} &=&\frac{%
OW_{n}-OW_{n-2}+OW_{n+1}-OW_{n-1}}{ab}  \notag \\
&&-\frac{A\gamma ^{\ast }\beta ^{2}-B\delta ^{\ast }\alpha ^{2}-ab\left(
A\gamma ^{\ast \ast }\beta -B\delta ^{\ast \ast }\alpha \right) }{\left(
ab\right) ^{2}}  \label{34}
\end{eqnarray}%
\begin{equation}
\left( ii\right) \sum_{r=0}^{n-1}OW_{2r}=\frac{OW_{2n}-OW_{2n-2}}{ab}-\frac{%
A\gamma ^{\ast }\beta ^{2}-B\delta ^{\ast }\alpha ^{2}}{\left( ab\right) ^{2}%
}\text{ \ \ \ \ \ \ }  \label{35}
\end{equation}%
\begin{equation}
\left( iii\right) \sum_{r=0}^{n-1}OW_{2r+1}=\frac{OW_{2n+1}-OW_{2n-1}}{ab}+%
\frac{A\gamma ^{\ast \ast }\beta -B\delta ^{\ast \ast }\alpha }{ab}.
\label{36}
\end{equation}
\end{theorem}

If we take the initial conditions $w_{0}=0$ and $w_{1}=1$ in the above
theorem, we obtain the results in \cite[Theorem 2.5]{yilmaz1}. If we take
the initial conditions $w_{0}=2$ and $w_{1}=b$ in the above theorem, we
obtain the results in \cite[Theorem 2.6]{yilmaz2}.

\section{Conclusion}

In this paper, we presented the generalized bi-periodic Fibonacci
quaternions $\left\{ W_{n}\right\} ,$ which is defined by $%
W_{n}=w_{n}e_{0}+w_{1}e_{1}+w_{2}e_{2}+w_{3}e_{3}$, where $%
w_{n}=aw_{n-1}+w_{n-2},$ if $n$ is even, $w_{n}=bw_{n-1}+w_{n-2},$ if $n$ is
odd with arbitrary initial conditions $w_{0},w_{1}$ and nonzero numbers $a,b$%
. By analogously, we defined the generalized bi-periodic Fibonacci octonions
and gave some basic properties of them. Our results not only gave a
generalization of the papers in \cite{chaos1, chaos2, yilmaz1, yilmaz2,
catarino}, but also included new results. The main contribution of this
research is one can get a great number of distinct quaternion and octonion
sequences by providing the initial values in the generalized bi-periodic
Fibonacci sequence $\left\{ w_{n}\right\} .$

\end{document}